\documentclass[11pt]{article}

\usepackage{cite}
\usepackage{amsmath,amssymb,amsfonts}
\usepackage{amsthm}
\usepackage{algorithmic}
\usepackage{graphicx}
\usepackage{textcomp}
\usepackage{braket}
\usepackage{xcolor}
\usepackage{todonotes}
\usepackage{hyperref}
\usepackage{a4wide}
\usepackage[ruled,vlined, linesnumbered]{algorithm2e}
\usepackage{authblk}

\def\BibTeX{{\rm B\kern-.05em{\sc i\kern-.025em b}\kern-.08em
    T\kern-.1667em\lower.7ex\hbox{E}\kern-.125emX}}

\newtheorem{thm}{Theorem}
\newtheorem{cor}{Corollary}

\newtheorem{rem}{Remark}
\newtheorem{obs}{Observation}

\newtheorem{assum}{Assumption}

\newcommand{\Rset}{\mathbb{R}}
\newcommand{\Xset}{\mathbb{X}}

\DeclareMathOperator{\E}{\mathbb{E}}

\begin{document}


 \title{A framework of distributionally robust possibilistic optimization} 

\author[1]{Romain Guillaume}
\author[2]{Adam Kasperski\footnote{Corresponding author}}
\author[2]{Pawe{\l} Zieli\'nski}

  \affil[1]{Universit{\'e} de Toulouse-IRIT Toulouse, France, \texttt{romain.guillaume@irit.fr}}
\affil[2]{
Wroc{\l}aw  University of Science and Technology, Wroc{\l}aw, Poland\\
            \texttt{\{adam.kasperski,pawel.zielinski\}@pwr.edu.pl}}

\date{}

\maketitle

\maketitle

 \begin{abstract}
 	In this paper,  an optimization problem with uncertain constraint coefficients is considered. Possibility theory is used to model uncertainty. Namely, a joint possibility distribution in constraint coefficient realizations, called scenarios, is specified. This possibility distribution induces a necessity measure in a scenario set, which in turn describes an ambiguity set of probability distributions in a scenario set. The distributionally robust approach is then used to convert the imprecise constraints into deterministic equivalents. Namely, the left-hand side of an imprecise constraint is evaluated by using a risk measure with respect to the worst probability distribution that can occur. In this paper, the Conditional Value at Risk is used as the risk measure, which
 generalizes the strict robust, and expected value approaches commonly used in literature.	
	A general framework 
for solving such a class of problems is described. 
Some cases which can be solved in polynomial time are identified.
 \end{abstract}


\section{Introduction}

Modeling uncertainty in optimization problems has been a subject of extensive research in recent decades.  This is due to the fact that in most applications, the exact values of input data in optimization models are unknown before a solution must be determined. Two basic questions arise while dealing with this uncertainty. The first is about the model of uncertainty adopted. The answer depends on the knowledge available about the input data. If the true probability distribution for uncertain parameters is known, then various stochastic models can be used (see, e.g.,~\cite{KM05}). For example, the uncertain constraints can be replaced with the so-called \emph{chance constraints}, in which we require that they be satisfied with a given probability~\cite{CC59}.
On the other hand, when only the set of possible realizations of the parameters (called a \emph{scenario set}) is available,  the  \emph{robust} approach is widely utilized (see, e.g.,~\cite{BN09}). In the \emph{strict robust} approach 
the constraints must be satisfied for all possible realizations of input data (scenarios). This strong requirement can be 
softened by introducing a \emph{budget}, i.e. some limit on the amount of uncertainty~\cite{BS04}.

In practice, an intermediate case often arises. Namely, partial information about the probability distributions is known, which leads to the so-called \emph{distributionally robust} models.  The left-hand side of an uncertain constraint is then evaluated by means of some risk measure with respect to the worst probability distribution that can occur.  Typically, the expected value is chosen as the risk measure. In this case, the decision-maker is risk neutral. Information about the distribution parameters, such as the mean or covariance matrix, can be reconstructed from statistical data, resulting in a class of \emph{data-driven problems}~\cite{BGK17}. On the other hand, some information about the distributions can be derived
 from
 the knowledge of experts or past experience. In this case, a
possibilistic uncertainty model can be appropriate (see, e.g.,~\cite{IR00}). 

In this paper, we will show how the \emph{possibility theory}~\cite{DP88} can be used in the context of distributionally robust optimization. Possibility theory offers a framework for dealing with uncertainty. In many practical situations, it is possible to say which realizations of the problem parameters are more likely to occur than others. Therefore, the common uncertainty sets used in robust optimization are too inaccurate. Also, focusing on the worst parameter realizations may lead to a large deterioration of the quality of the computed solutions (the large price of robustness). On the other hand, due to the lack of statistical data or the complex nature of the world, precise probabilistic information often cannot be provided. Then, the possibility theory, whose axioms are less rigorous than the axioms of probability theory, can be used to handle uncertainty.
In this paper, we will use some relationships between possibility and probability theories, which allows us to define a class of admissible probability distributions for the problem parameters. Namely, a possibility distribution in the set of parameter realizations induces a \emph{necessity measure} in this set, which can be seen as a lower bound on the probability measure~\cite{DD06}. 
We will further take into account the risk aversion of decision-makers by using a risk measure called \emph{Conditional Value at Risk} (CVaR for short)~\cite{RU00}.  CVaR is a convex coherent risk measure that possesses some favorable computational properties.
Using it, we can generalize the strict robust, and expected value-based approaches.
 In this paper, we show a general framework for solving a problem with uncertain parameters  
 under the assumed model of uncertainty based on the above possibility-probability relationship. We also propose a method of constructing a joint possibility distribution for the problem parameters, which leads to tractable reformulations that can be solved efficiently by some standard off-the-shelf solvers. This paper is an extended version of~\cite{GKZ21a}, where one particular uncertainty model was discussed.
 
The approach considered in this paper belongs to the area of \emph{fuzzy optimization} in which many solution concepts have been proposed in the literature (see, e.g.,~\cite{LK10}). Generally, fuzzy sets can be used to model uncertainty or preferences in optimization models. In the former case (used in this paper), the membership function of a fuzzy set is interpreted as a possibility distribution for the values of some uncertain quantity~\cite{DP88}. This interpretation leads to various solution concepts (see, e.g.~\cite{LK10,LH92} for surveys). For example, fuzzy chance constraints can be used in which a constraint with fuzzy coefficients should be satisfied with the highest degree of possibility or necessity~\cite{LIU01}. Furthermore, the uncertain objective function with fuzzy coefficients can be replaced with a deterministic equivalent by using some defuzzification methods~\cite{PZT11}. Some recent applications of fuzzy robust optimization methods can be found, for example, in~\cite{LY19, PRB21}. Our approach is new in this area. It uses a link between fuzzy and stochastic models established by the possibility theory. This allows us to use well-known stochastic methods under the fuzzy (possibilistic) model of uncertainty. Our model is illustrated with two examples, namely continuous knapsack and portfolio selection problems. Its validation on real problems is an interesting area of further research.

This paper is organized as follows. In Section~\ref{sec1}, we consider an uncertain constraint and present various methods of converting such a constraint into deterministic equivalents. In Sections~\ref{sec2} and~\ref{sec3}, we recall basic notions of possibility theory. In particular, we describe some relationships between possibility and probability theories, which allow us to build a model of uncertainty in which a joint possibility distribution is specified in the set of scenarios. In Section~\ref{stfc} we discuss some
tractable cases of our model, which can be converted into linear or second-order cone programming problems. 
In Section~\ref{sec4}, we characterize the complexity of basic classes of optimization problems, such as linear programming or combinatorial ones. Finally, in Section~\ref{sec5}, we illustrate our concept with two computational examples -  the continuous knapsack and portfolio selection problems.

\section{Handling uncertain constraints}
\label{sec1}

Consider the following \emph{uncertain linear constraint}:
\begin{equation}
	\label{uc}
	\tilde{\pmb{a}}^T\pmb{x}\leq b,
\end{equation}
where $\pmb{x}\in \mathbb{X}$ is a vector of decision variables with a domain $\mathbb{X}\subseteq \mathbb{R}^n$, 
$\tilde{\pmb{a}}=(\tilde{a}_j)_{j\in [n]}$ is a vector of uncertain parameters, $\tilde{a}_j$ is 
a real-valued uncertain variable
 and $b\in \Rset$ is a prescribed bound.
 A particular realization $\pmb{a}=(a_j)_{j\in[n]}\in \Rset^n$ of $\tilde{\pmb{a}}$ is called a \emph{scenario}
 and the set of all possible scenarios, denoted by $\mathcal{U}$, is called the \emph{scenario set} or \emph{uncertainty set}.
 Intuitively, the uncertain constraint~(\ref{uc}) means that  $\pmb{a}^T\pmb{x}\leq b$ is satisfied by 
 a solution $\pmb{x}\in \mathbb{X}$  of an optimization problem for some or all scenarios~$\pmb{a}\in \mathcal{U}$,
 depending on a method of handling the uncertain constraints.

Let $\mathcal{P}(\mathcal{U})$ be the set of all probability distributions in~$\mathcal{U}$. In this paper, we assume that $\tilde{\pmb{a}}$ is a random vector with probability distribution in $\mathcal{P}(\mathcal{U})$. The true probability distribution ${\rm P}$ for $\tilde{\pmb{a}}$ can be unknown or only partially known. The latter case can be expressed by providing an \emph{ambiguity set of probability distributions} $\mathcal{P}\subseteq \mathcal{P}(\mathcal{U})$, 
  so that ${\rm P}$ is only known to belong to~$\mathcal{P}$ with high confidence.
     There are many ways of constructing the set~$\mathcal{P}$. Most of them are based on statistical data available for $\tilde{\pmb{a}}$, which enables us to estimate the moments (expected values or covariance matrix) or the distribution shape (see, e.g.,~\cite{DY10,WKS14,EK18}). In the absence of statistical data, one can try 
   to estimate $\mathcal{P}$ using experts' opinions. In Sections~\ref{sec2} and~\ref{sec3}
   we will show how to construct an ambiguity set of probability distributions~$\mathcal{P}$ by
   using a link between fuzzy and stochastic models established by possibility theory.

In order to convert~(\ref{uc}) into a deterministic equivalent, one can use the following 
\emph{distributionally robust constraint}~\cite{WKS14, DY10}:
 \begin{equation}
	\label{uce}
	\sup_{{\rm P}\in \mathcal{P}}\E_{{\rm P}}[\tilde{\pmb{a}}^T\pmb{x}]\leq b
\end{equation}
or equivalently 
\[
\E_{{\rm P}}[\tilde{\pmb{a}}^T\pmb{x}]\leq b \; \forall\, {\rm P}\in \mathcal{P},
\]
where $\E_{{\rm P}}[\cdot]$ is the expected value with respect to the probability distribution ${\rm P}\in \mathcal{P}$.  If $\mathcal{P}=\mathcal{P}(\mathcal{U})$, then~(\ref{uce}) reduces to the following \emph{strict robust constraint}~\cite{BN09}:
 \begin{equation}
	\label{uce1}
	\sup_{\pmb{a}\in \mathcal{U}} \pmb{a}^T\pmb{x}\leq b.
\end{equation}
Indeed, if $\mathcal{P}=\mathcal{P}(\mathcal{U})$, then probability equal to~1 is assigned to scenario $\pmb{a}\in \mathcal{U}$ maximizing $\pmb{a}^T\pmb{x}$ (i.e. to a worst scenario for $\pmb{x}$). The strict robust constraint is known to be very conservative, as it can lead to a large deterioration of the quality of the solutions computed. One can limit the number of scenarios by adding a budget to $\mathcal{U}$~\cite{BS04} to overcome this drawback. This budget narrows the set $\mathcal{U}$ to scenarios that are at some prescribed distance to a given \emph{nominal scenario} $\hat{\pmb{a}}\in \mathcal{U}$.  We thus assume that scenarios far from $\hat{\pmb{a}}$ are unlikely to occur.

We can generalize~(\ref{uce}) by introducing a convex non-decreasing function $g:\Rset\rightarrow \Rset$, and considering the following constraint:
 \begin{equation}
	\label{uceu}
	\sup_{{\rm P}\in \mathcal{P}}\E_{{\rm P}}[g(\tilde{\pmb{a}}^T\pmb{x})]\leq b.
\end{equation}
For example, the function~$g$ can be interpreted as a \emph{disutility} for the values of $\pmb{a}^T\pmb{x}$, where $\pmb{a}$ is a scenario in $\mathcal{U}$. In particular,~$g$ can be of the form $g(y)=y$, in which case we get~(\ref{uce}).
Notice that $g(\tilde{\pmb{a}}^T\pmb{x})$ is a function of  the random variable $\tilde{\pmb{a}}^T\pmb{x}$, so it is 
also
 a random variable.

By using the expectation in~(\ref{uceu}), one assumes risk neutrality of decision-makers. 
In order to take an individual decision-maker's risk aversion into account, one can replace the expectation with the \emph{Conditional Value At Risk} (called also \emph{Expected Shortfall})~\cite{RU00}. If ${\rm X}$ is a real random variable, then
the Conditional Value At Risk with respect to the probability distribution ${\rm P}\in \mathcal{P}$ is defined as follows:
\begin{equation}
\label{defcvar}
{\rm CVaR}_{\mathrm{P}}^{\epsilon}[{\rm X}]=\inf_{t\in\Rset}(t+\frac{1}{1-\epsilon}\E_{{\rm P}}[{\rm X}-t]_{+}),
\end{equation}
where  $\epsilon\in [0,1)$ is a given risk level and $[x]_{+}=\max\{0,x\}$. Notice that ${\rm CVaR}_{\mathrm{P}}^{0}[{\rm X}]={\rm E}_{\mathrm{P}}[{\rm X}]$. On the other hand, if $\epsilon\rightarrow 1$, then ${\rm CVaR}_{\mathrm{P}}^{\epsilon}[{\rm X}]$ is equal to the maximal value that ${\rm X}$ can take. We will consider the following constraint:
\begin{equation}
\sup_{{\rm P}\in \mathcal{P}} {\rm CVaR}_{\mathrm{P}}^{\epsilon}[g(\tilde{\pmb{a}}^T\pmb{x})]\leq b
\label{uccvar1}
\end{equation}
or equivalently, by using~(\ref{defcvar}):
 \begin{equation}
	\label{uccvar}
	\sup_{{\rm P}\in \mathcal{P}} \inf_{t\in \Rset}(t+\frac{1}{1-\epsilon}\E_{{\rm P}}[g(\tilde{\pmb{a}}^T\pmb{x})-t]_{+})\leq b.
\end{equation}
The parameter $\epsilon\in [0,1)$ models risk aversion of
decision-makers. If $\epsilon=0$, then~(\ref{uccvar}) is equivalent to~(\ref{uceu}). On the other hand, when $\epsilon\rightarrow 1$, then~(\ref{uccvar}) becomes
\begin{equation}
	\label{uce3}
	\sup_{\pmb{a}\in \mathcal{U'}} g(\pmb{a}^T\pmb{x})\leq b,
\end{equation}
where $\mathcal{U}'\subseteq \mathcal{U}$ is the subset of scenarios that can occur with a positive probability for some distribution in $\mathcal{P}$. Thus, the approach  based on 
the Conditional Value At Risk generalizes~(\ref{uce1}) and~(\ref{uceu}).
Moreover,
the CVaR criterion can be used to establish a conservative approximation of the chance constraints~\cite{NS06}, namely
\begin{equation}
\label{cvarappr}
{\rm CVaR}_{\mathrm{P}}^{\epsilon}[g(\tilde{\pmb{a}}^T\pmb{x})]\leq b \Rightarrow {\rm P}(g(\tilde{\pmb{a}}^T\pmb{x})\leq b)\geq \epsilon.
\end{equation}
 Using condition~(\ref{cvarappr}) we conclude that a solution feasible to~(\ref{uccvar1}) is also feasible to the following  \emph{robust chance constraint}:
\begin{equation}
\label{chancec}
\inf_{{\rm P}\in \mathcal{P}} {\rm P}(g(\tilde{\pmb{a}}^T\pmb{x})\leq b)\geq \epsilon.
\end{equation}
Optimization models with~(\ref{chancec}) can be hard to solve.
 They are typically nonconvex, whereas the models with~(\ref{uccvar1}) can lead to more tractable optimization problems.

\section{Possibility distribution-based model for uncertain constraints}
\label{sec2}

Let $\Omega$ be a set of alternatives. A primitive object of possibility theory (see, e.g.,~\cite{BD06}) is a \emph{possibility distribution} $\pi: \Omega\rightarrow [0,1]$, which assigns to each element $u\in \Omega$ a \emph{possibility degree} $\pi(u)$. We only assume that $\pi$ is \emph{normal}, i.e. there is $u\in \Omega$ such that $\pi(u)=1$. The possibility distribution can be built by using available data or experts' opinions (see, e.g.,~\cite{DP88}). A possibility distribution $\pi$ induces the following \emph{possibility} and \emph{necessity measures} in~$\Omega$:
\begin{align}
\Pi(A)&=\sup_{u\in A} \pi(u),\; A \subseteq \Omega, \label{possdef}\\
{\rm N}(A)&=1-\Pi(A^c)=1-\sup_{u\in A^c} \pi(u),\; A\subseteq \Omega, \label{necdef}
\end{align}
where $A^c=\Omega\setminus A$ is the complement of event~$A$.
In this paper we assume that $\pi$ represents uncertainty in $\Omega$, i.e. some knowledge about the uncertain quantity $\tilde{u}$ taking values in $\Omega$. 
We now recall, following~\cite{DD06}, a probabilistic interpretation of the pair $[\Pi, {\rm N}]$ induced by the possibility distribution $\pi$.
 Define 
 \begin{equation}
 \mathcal{P}_\pi=\{{\rm P}: \forall A \text{ measurable } {\rm N}(A)\leq {\rm P}(A)\}
                         =\{{\rm P}: \forall A \text{ measurable } \Pi(A)\geq {\rm P}(A)\}.\label{defpm}
 \end{equation} 
 In this case $\sup_{{\rm P}\in \mathcal{P}_\pi} {\rm P}(A)=\Pi(A)$ and $\inf_{{\rm P}\in \mathcal{P}_\pi} {\rm P}(A)={\rm N}(A)$ (see~\cite{DD06, DP92, CA99}).  Therefore, the possibility distribution~$\pi$ in $\Omega$ encodes a family of probability measures in~$\Omega$. Any probability measure ${\rm P}\in \mathcal{P}_\pi$ is said to be \emph{consistent with}~$\pi$ and for any event $A\subseteq \Omega$ the inequalities
\begin{equation} 
\label{eqnpp}
 {\rm N}(A)\leq {\rm P}(A)\leq \Pi(A)
 \end{equation}
 hold.
  A detailed discussion on the expressive power of~(\ref{eqnpp}) can also be found in~\cite{TMD13}.  Using the assumption that $\pi$ is normal, one can show that $\mathcal{P}_\pi$ is not empty. In fact, the probability distribution such that ${\rm P}(\{u\})=1$ for some $u\in \Omega$ with $\pi(u)=1$ is in $\mathcal{P}_\pi$. To see this, consider two cases. If $u\notin A$, then ${\rm N}(A)=0$ and ${\rm P}(A)\geq {\rm N}(A)=0$. If $u\in A$, then ${\rm P}(A)=1\geq {\rm N}(A)$. 
    
Assume that $\pi:\mathbb{R}^n\rightarrow [0,1]$ is a \emph{joint possibility distribution} for the vector of uncertain parameters~$\tilde{\pmb{a}}$ in the constraint~(\ref{uc}). The value of $\pi(\pmb{a})$, $\pmb{a}\in \mathbb{R}^n$, is the possibility degree for scenario $\pmb{a}$.
  A detailed method of constructing $\pi$ will be shown in Section~\ref{sec3}. By~(\ref{defpm}), the possibility distribution $\pi$ induces
  an ambiguity set~$\mathcal{P}_\pi$ of probability distributions
  for~$\tilde{\pmb{a}}$, consistent with~$\pi$. 
  Define
  \begin{equation} 
  \mathcal{C}(\lambda)=\{\pmb{a}\in \mathbb{R}^n: \pi(\pmb{a})\geq \lambda\}, \; \lambda\in (0,1].
  \label{ccut}
  \end{equation}
  The set $\mathcal{C}(\lambda)$ contains all scenarios $\pmb{a}\in \mathbb{R}^n$, whose possibility of occurrence is at least $\lambda$.
   We will also use $\mathcal{C}(0)$ to denote the \emph{support} of $\pi$, i.e the smallest  subset (with respect to inclusion) 
   of $\Rset^n$ such that ${\rm N}(\mathcal{C}(0))=1$. By definition~(\ref{ccut}), $\mathcal{C}(\lambda)$, $\lambda\in [0,1]$, is a family of nested sets, i.e. $\mathcal{C}(\lambda_2)\subseteq \mathcal{C}(\lambda_1)$ if $0\leq \lambda_1\leq \lambda_2\leq 1$. We will make the following, not particularly restrictive, assumptions about~$\pi$:
  \begin{assum} \;
  \label{assump1}
  \begin{enumerate}
	\item There is a scenario $\hat{\pmb{a}}\in \Rset^n$ such that $\pi(\hat{\pmb{a}})=1$.
  	\item $\pi$ is continuous in $\mathcal{C}(0)$.
	\item The support $\mathcal{C}(0)$ of $\pi$ coincides with the scenario set~$\mathcal{U}$ and is a compact set.
  	\item The sets $\mathcal{C}(\lambda)$ for $\lambda\in [0,1]$ are closed and convex.
	\item The sets $\mathcal{C}(\lambda)$ for $\lambda\in [0,1)$ have non-empty interiors.
	\end{enumerate}
  \end{assum}
Assumption~\ref{assump1}.1 is a normalization one that must be satisfied by any possibility distribution. There is at least one scenario~$\hat{\pmb{a}}$, called a \emph{nominal scenario}, whose possibility of occurrence is equal to~1. Assumption~\ref{assump1}.2 will allow us to represent the ambiguity set $\mathcal{P}_\pi$ as a family of nested confidence sets. Assumption~\ref{assump1}.3 means that the compact set $\mathcal{C}(0)$ contains all possible scenarios, i.e. realizations of the uncertain vector~$\tilde{\pmb{a}}$, so it coincides with the scenario set $\mathcal{U}$. The last two assumptions will allow us to construct some linear or second-order cone reformulations of the constraint~(\ref{uccvar1}).

The assumptions about $\pi$ imply the following equation:
 $${\rm N}(\mathcal{C}(\lambda))=1-\lambda, \; \lambda\in [0,1].$$
 So, the degree of necessity (certainty) of~$\mathcal{C}(\lambda)$ (event ``any $\pmb{a}\in \mathcal{C}(\lambda)$'')
  equals $1-\lambda$.
 Hence, $\mathcal{C}(\lambda)$ is a confidence set, with the confidence (necessity) degree equal to $1-\lambda$.
 The continuity of $\pi$ allows us also to establish the following  formulation of the ambiguity set $\mathcal{P}_\pi$ (see~\cite{GKZ21b}):
$$\mathcal{P}_\pi=\{{\rm P}\in \mathcal{P}(\mathcal{U}) : {\rm P}(\mathcal{C}(\lambda))\geq {\rm N}(\mathcal{C}(\lambda))=1-\lambda, \lambda\in [0,1]\}.$$
Thus $\mathcal{P}_\pi$ is represented 
by a family of confidence sets $\mathcal{C}(\lambda)$, $ \lambda\in [0,1]$, with confidence (necessity) degrees 
$1-\lambda$.	

In order to construct a tractable reformulation of the constraint~(\ref{uccvar1}), we will use the following approximation of $\mathcal{P}_\pi$. Let $\ell$ be a positive integer and fix $\Lambda=\{0,1,\dots,\ell\}$.  Define $\lambda_i=i/\ell$ for $i\in \Lambda$. Consider the following ambiguity set:
\begin{equation}
\mathcal{P}^{\ell}_\pi=\{{\rm P}\in \mathcal{P}(\mathcal{U}) : {\rm P}(\mathcal{C}(\lambda_i))\geq {\rm N}(\mathcal{C}(\lambda_i))=
1-\lambda_i, i\in \Lambda\}.
\label{asi}
\end{equation}
Set $\mathcal{P}^{\ell}_\pi$ is a discrete approximation of $\mathcal{P}_\pi$. It is easy to verify that $\mathcal{P}_\pi\subseteq \mathcal{P}^{\ell}_\pi$ for any $\ell> 0$ (if ${\rm P}\in \mathcal{P}_{\pi}$, then ${\rm P}\in \mathcal{P}^{\ell}_{\pi}$). By fixing sufficiently large constant $\ell$, we obtain an arbitrarily close approximation of $\mathcal{P}_\pi$. 
The following theorem is the key result of the paper:
\begin{thm}
\label{tuccvar2}
	The constraint~(\ref{uccvar1}) for the ambiguity set~$\mathcal{P}=\mathcal{P}^{\ell}_\pi$, i.e.
\begin{equation}
\sup_{{\rm P}\in \mathcal{P}^{\ell}_\pi} {\rm CVaR}_{\mathrm{P}}^{\epsilon}[g(\tilde{\pmb{a}}^T\pmb{x})]\leq b,
\label{uccvar2}
\end{equation}
is equivalent to the following semi-infinite system of constraints:
\begin{align}
&w+\sum_{i\in \Lambda}(\lambda_i-1) v_i\leq (b-t)(1-\epsilon)&\label{repruccvar1}\\
&w-\sum_{j\leq i}  v_j\geq 0 &\forall  i\in \Lambda \label{repruccvar2}\\
&w-\sum_{j\leq i}  v_j+t \geq  g(\pmb{a}^T\pmb{x}) \;\;\;  \forall \pmb{a} \in \mathcal{C}(\lambda_i),
				 &\forall  i\in \Lambda \label{repruccvar3}\\
		&v_i\geq 0 &\forall  i\in \Lambda\label{repruccvar4}\\
	&w,t\in \Rset.&	\label{repruccvar5}
\end{align}
\end{thm}  
  \begin{proof}
Using~(\ref{uccvar}) for
  $\mathcal{P}=\mathcal{P}^{\ell}_\pi$, we can express the left-hand side of~(\ref{uccvar2}) 
  as follows:
  $$\max_{{\rm P}\in \mathcal{P}^{\ell}_\pi} \min_{t\in \Rset}\left (t+\frac{1}{1-\epsilon}\int_{\mathcal{C}(\lambda_0)} [g(\pmb{a}^T\pmb{x})-t]_{+} \; {\rm d} \, \mathrm{P}(\pmb{a})\right).$$
  The function $t+\frac{1}{1-\epsilon}\int_{\mathcal{C}(\lambda_0)} [g(\pmb{a}^T\pmb{x})-t]_{+} \; {\rm d} \, \mathrm{P}(\pmb{a})$ is real-valued, linear in $\mathrm{P}$ and convex in $t$ (see~\cite{DY10}). Hence, by the minimax theorem, it is equivalent to
    $$\min_{t\in \Rset}  \left(t+\frac{1}{1-\epsilon}\max_{{\rm P}\in \mathcal{P}^{\ell}_\pi}\int_{\mathcal{C}(\lambda_0)} [g(\pmb{a}^T\pmb{x})-t]_{+} \; {\rm d}\, \mathrm{P}(\pmb{a})\right).$$
  In consequence, the inequality~(\ref{uccvar2}) holds if and only if there is $t$ such that
    \begin{equation}
    \label{e00} 
     \max_{{\rm P}\in \mathcal{P}^{\ell}_\pi}\int_{\mathcal{C}(\lambda_0)} [g(\pmb{a}^T\pmb{x})-t]_{+} \; {\rm d}\, \mathrm{P}(\pmb{a})\leq (b-t)(1-\epsilon).
     \end{equation}
     Taking into account the form of the ambiguity set~$\mathcal{P}^{\ell}_\pi$ (see~(\ref{asi}))
     the left-hand side of~(\ref{e00}) can be expressed as the following \emph{problem of moments}:
  \begin{equation}
\label{mod10}
	\begin{array}{lll}
		\max &  \displaystyle \int_{\mathcal{C}(\lambda_0)} [g(\pmb{a}^T\pmb{x})-t]_{+} \; {\rm d}\, \mathrm{P}(\pmb{a}) \\
			& \displaystyle \int_{\mathcal{C}(\lambda_0)} \pmb{1}_{[\pmb{a}\in \mathcal{C}(\lambda_i)]} \; {\rm d}
			\, \mathrm{P}(\pmb{a})\geq 1-\lambda_i &\forall  i\in \Lambda\\
			& \displaystyle  \int_{\mathcal{C}(\lambda_0)} {\rm d} \, \mathrm{P}(\pmb{a})=1\\
			&  \mathrm{P} \in \mathcal{M}_{+}(\mathbb{R}^n),
	\end{array}
\end{equation}
where $\pmb{1}_{[A]}$ denotes the indicator function of an event~$A$
and
$\mathcal{M}_{+}(\mathbb{R}^n)$ is the set of all non-negative measures on~$\mathbb{R}^n$.  
The dual of  the problem of moments takes the following form (see, e.g.,~\cite{KL19}):
\begin{equation}
\label{mod20}
	\begin{array}{llll}
			\min  & \displaystyle w+\sum_{i\in \Lambda}(\lambda_i-1) v_i\\
				& \displaystyle w-\sum_{i\in \Lambda} \pmb{1}_{[\pmb{a}\in \mathcal{C}(\lambda_i)]} v_i\geq [g(\pmb{a}^T\pmb{x})-t]_{+} & \forall \pmb{a} \in \mathcal{C}(\lambda_0) \\
				& v_i\geq 0 &\forall  i\in \Lambda\\
				& w\in \Rset, &	
\end{array}
\end{equation}
where $v_i$, $i\in \Lambda$,  and $w$ are dual variables that correspond to
the constraints of the primal problem~(\ref{mod10}).
Strong duality
 implies that the optimal objective values of~(\ref{mod10}) and (\ref{mod20}) are the same
(see, e.g.,~\cite[Theorem~1, Corollary~1]{KL19}).
Define 
$$\overline{\mathcal{C}}(\lambda_i)=\mathcal{C}(\lambda_i)\setminus \mathcal{C}(\lambda_{i+1}),\; i\in\{0,\dots,\ell-1\}$$
and $\overline{\mathcal{C}}(\lambda_{\ell})=\mathcal{C}(\lambda_{\ell})$.
Hence $\overline{C}(\lambda_i)$, $i\in \Lambda$, form a partition of support~$\mathcal{C}(\lambda_0)$ into $\ell+1$ disjoint sets.
The model~(\ref{mod20}) can then be rewritten as follows:
\begin{equation}
\label{mod30}
	\begin{array}{llll}
			\min  & \displaystyle w+\sum_{i\in \Lambda}(\lambda_i-1) v_i\\
				& \displaystyle w-\sum_{j\in \Lambda} \pmb{1}_{[\pmb{a}\in \mathcal{C}(\lambda_j)]} v_j\geq  [g(\pmb{a}^T\pmb{x})-t]_{+}  & \forall \pmb{a} \in \overline{\mathcal{C}}(\lambda_i), 
				&  \forall  i\in \Lambda \\
				& v_i\geq 0 &&\forall  i\in \Lambda\\
				&w\in \Rset,&
	\end{array}
\end{equation}
which is equivalent to
\begin{equation}
\label{mod40}
	\begin{array}{llll}
			\min  & \displaystyle w+\sum_{i\in \Lambda}(\lambda_i-1) v_i\\
				& \displaystyle w-\sum_{j\leq i}  v_j\geq  [g(\pmb{a}^T\pmb{x})-t]_{+}  & \forall \pmb{a} \in \overline{\mathcal{C}}(\lambda_i), &\forall   i\in \Lambda \\
				& v_i\geq 0 && \forall   i\in \Lambda\\
				& w\in \Rset &
	\end{array}
\end{equation}
The model~(\ref{mod40}) can be rewritten as
\begin{equation}
\label{mod50}
	\begin{array}{llll}
			\min  & \displaystyle w+\sum_{i\in \Lambda}(\lambda_i-1) v_i\\
				& \displaystyle w-\sum_{j\leq i}  v_j\geq 0 & &\forall   i\in \Lambda \\
				& \displaystyle w-\sum_{j\leq i}  v_j+t \geq  g(\pmb{a}^T\pmb{x})  & \forall \pmb{a} \in \overline{\mathcal{C}}(\lambda_i), &\forall   i\in \Lambda \\
				& v_i\geq 0 && \forall   i\in \Lambda\\
				& w\in \Rset &
	\end{array}
\end{equation}
The semi-infinite system of constraints in~(\ref{mod50}) can be represented as
\begin{equation}
\label{mod60}
\max_{\pmb{a} \in \overline{\mathcal{C}}(\lambda_i)} g(\pmb{a}^T\pmb{x}) \leq w-\sum_{j\leq i}  v_j+t.
\end{equation}
Because $g$ is convex, the maximum in~(\ref{mod60}) is attained at the boundary of $\overline{\mathcal{C}}(\lambda_i)$ which coincides  with the boundary of $\mathcal{C}(\lambda_i)$. In consequence, we can replace $\overline{\mathcal{C}}(\lambda_i)$ with 
$\mathcal{C}(\lambda_i)$ in~(\ref{mod50}), which leads to~(\ref{repruccvar1})-(\ref{repruccvar5}).
\end{proof}

\begin{rem}
\label{rtfc}
The tractability of the formulation~(\ref{repruccvar1})-(\ref{repruccvar5}), and thus~(\ref{uccvar2}), depends on a method of handling the semi-infinite system of constraints~(\ref{repruccvar3}). Notice that, by Assumption~\ref{assump1},  this system is equivalent to
 \begin{equation}
\label{sinfc}
 w-\sum_{j\leq i}  v_j+t\geq \max_{\pmb{a} \in\mathcal{C}(\lambda_i)} g(\pmb{a}^T\pmb{x}), \;\; i\in \Lambda.
\end{equation}
Deciding efficiently if  (\ref{sinfc}) is satisfied depends on coping with the optimization problem on the right-hand side of~(\ref{sinfc}), which, in turn, depends 
on the forms of confidence set  $\mathcal{C}(\lambda_i)$ induced by the possibility distribution~$\pi$ (see~(\ref{ccut})), and function~$g$.
In Sections~\ref{sec3}  and~\ref{stfc}, we will be concerned with
the abovementioned topic.
\end{rem}

\section{Constructing a joint possibility distribution for an uncertain constraint}
\label{sec3}

In this section, we will propose a method to construct a joint possibility distribution $\pi:\mathbb{R}^n\rightarrow[0,1]$ for 
the vector of uncertain parameters~$\tilde{\pmb{a}}$ in the constraint~(\ref{uc}). 
We will also provide a decomposition of~$\pi$ into a family of confidence sets, according to~(\ref{ccut}), which will allow us to deal with~(\ref{sinfc}) and, consequently, with the formulation~(\ref{repruccvar1})-(\ref{repruccvar5}).

Choose a component~$\tilde{a}_j$  of $\tilde{\pmb{a}}=(\tilde{a}_1,\dots,\tilde{a}_n)$, where $\tilde{a}_j$ is 
a real-valued uncertain variable.
 Suppose that the true value of $\tilde{a}_j$ is known to belong to the closed interval $[\hat{a}_j-\underline{a}_j, \hat{a}_j+\overline{a}_j]$, where $\hat{a}_j\in \Rset$ is the nominal (expected) value of $\tilde{a}_j$ and $\underline{a}_j, \overline{a}_j\in \Rset_{+}$ represent the maximum deviations of the value of $\tilde{a}_j$ from $\hat{a}_j$.
The hyperrectangle 
$$\mathcal{I}=[\hat{a}_1-\underline{a}_1,\hat{a}_1+\overline{a}_1]\times\dots\times [\hat{a}_n-\underline{a}_n,\hat{a}_n+\overline{a}_n]$$
 (the Cartesian product of the intervals) contains all possible values of $\tilde{\pmb{a}}$. In practice, $\mathcal{I}$ is often an overestimation of the scenario set $\mathcal{U}$, which is due to correlations among the components of $\tilde{\pmb{a}}$. In particular, it can be highly improbable that all components of $\tilde{\pmb{a}}$ will take extreme values at the same time~\cite{BS04}. Let $\delta(\pmb{a})$ be a distance of scenario $\pmb{a}$ from the nominal scenario $\hat{\pmb{a}}$. We will use the following distance function:
\begin{equation}
\label{disteq}
\delta(\pmb{a})=||\pmb{B}(\pmb{a}-\hat{\pmb{a}})||_{p},
\end{equation}
where $||\cdot||_{p}$ is the $L_p$-norm, $p\geq 1$, and $\pmb{B}\in \mathbb{R}^{m\times n}$ is a given matrix which can be used to model some interactions among the components of $\tilde{\pmb{a}}$. In the simplest case, $\pmb{B}$ is an identity $n\times n$ matrix and $\delta(\pmb{a})=||(\pmb{a}-\hat{\pmb{a}})||_{p}$. Let $\overline{\delta}\geq 0$ be a parameter, called a \emph{budget}, which denotes the maximum possible deviation of the values of $\tilde{\pmb{a}}$ from $\hat{\pmb{a}}$.  It stands for a limit on the amount of uncertainty for~$\tilde{\pmb{a}}$. The idea of the budgeted uncertainty has been proposed in~\cite{BS04}.
Let 
$$\mathcal{D}=\{\pmb{a}\in \mathbb{R}^n: \delta(\pmb{a})\leq \overline{\delta}\}$$
be the set of all scenarios whose distance from the nominal scenario does not exceed $\overline{\delta}$. 
We will assume that the scenario set is defined as $\mathcal{U}=\mathcal{I}\cap \mathcal{D}$.
Hence $\Pi(\mathcal{U}^c)=0$ and, by~(\ref{necdef}), ${\rm N}(\mathcal{U})=1$.
 \begin{figure}[ht!]
\centering
\includegraphics[width=12cm]{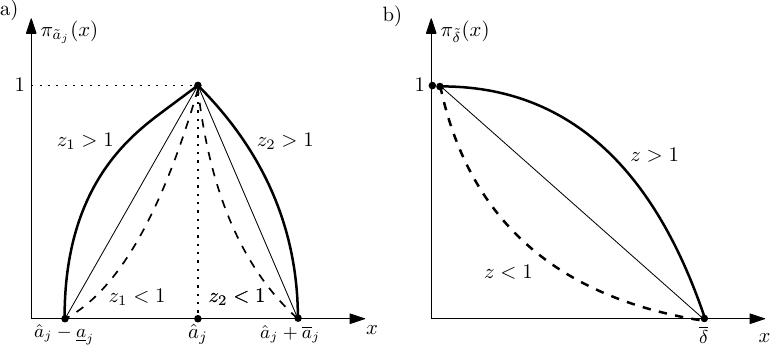}
\caption{Membership functions, $\mu_{\tilde{a}_j}$ and $\mu_{\tilde{\delta}}$, of fuzzy intervals $\tilde{a}_j=\braket{\hat{a}_j,\underline{a}_j,\overline{a}_j}_{z_1-z_2}$
 and $\tilde{\delta}=\braket{0, \overline{\delta}}_{z}$ modeling the possibility distributions for $\tilde{a}_j$ and $\tilde{\delta}$,
 respectively.} \label{fig1}
\end{figure}

Without any additional information, the scenario set $\mathcal{U}$ induces the following joint possibility distribution for $\tilde{\pmb{a}}$: $\pi(\pmb{a})=1$ if $\pmb{a}\in \mathcal{U}$ and $\pi(\pmb{a})=0$, otherwise. In this case, $\mathcal{C}(\lambda)=\mathcal{U}$ for each $\lambda\in [0,1]$ and $\mathcal{P}_{\pi}$ (also $\mathcal{P}^{\ell}_{\pi}$)  contains all probability distributions in $\mathcal{U}$. Constraint~(\ref{uccvar2}) is then equivalent to the strict robust constraint~(\ref{uce3}) with $\mathcal{U}'=\mathcal{U}$.

We will now extend this idea to refine the model of uncertainty. We will make a reasonable assumption that scenarios closer to the nominal scenario $\hat{\pmb{a}}$ will be more possible than those with a greater distance from $\hat{\pmb{a}}$.
Let $\pi_{\tilde{a}_j}$ be a possibility distribution for uncertain parameter~$\tilde{a}_j$ with the support $[\hat{a}_j-\underline{a}_j, \hat{a}_j+\overline{a}_j]$ and $\pi_{\tilde{a}_j}(\hat{a}_j)=1$. The function $\pi_{\tilde{a}_j}$ is continuous, strictly increasing in $[\hat{a}_j-\underline{a}_j]$ and strictly decreasing in $[\hat{a}_j+\underline{a}_j]$. 
Following~\cite{DP88},
one can identify $\pi_{\tilde{a}_j}$ with a membership function~$\mu_{\tilde{a}_j}$ of a fuzzy interval in $\mathbb{R}$
that models~$\tilde{a}_j$, namely $\pi_{\tilde{a}_j}=\mu_{\tilde{a}_j}$. We denote it by 
$\tilde{a}_j=\braket{\hat{a}_j,\underline{a}_j,\overline{a}_j}_{z_1-z_2}$ and 
its membership function~$\mu_{\tilde{a}_j}$
 can be of the following form (see Figure~\ref{fig1}a):
\begin{equation}
	\mu_{\tilde{a}_j}(x)=\left\{\begin{array}{lllll}
				(1+\frac{x-\hat{a}_j}{\underline{a}_j})^{\frac{1}{z_1}} & \text{for} \;x\in [\hat{a}_j-\underline{a}_j,\hat{a}_j]\\
				(1+\frac{\hat{a}_j-x}{\overline{a}_j})^{\frac{1}{z_2}} & \text{for} \; x\in [\hat{a}_j,\hat{a}_j+\overline{a}_j]\\
				0 & \text{otherwise},
			\end{array}\right.
\label{mua}
\end{equation}
where $z_1>0$ and $ z_2>0$ are the shape parameters. One can also use 
 fuzzy intervals of the $L-R$ type~\cite{DP78} to model~$\tilde{a}_j$.
Note that,
 for large $z_1,z_2$, fuzzy interval~$\tilde{a}_j$ tends to the traditional interval $[\hat{a}_j-\overline{a}_j, \hat{a}_j+\overline{a}_j]$. On the other hand, when  $z_1,z_2$ are close to~0, $\tilde{a}_j$ tends to the nominal value $\hat{a}_j$. Therefore, the shape parameters can be used to model the amount of uncertainty for~$\tilde{a}_j$.

Similarly, let $\pi_{\tilde{\delta}}$ be a possibility distribution for the deviation such that $\pi_{\tilde{\delta}}(0)=1$, $\pi_{\tilde{\delta}}$ is continuous strictly decreasing in $[0,\overline{\delta}]$ and $\pi_{\tilde{\delta}}(\delta)=0$ for $\delta\leq 0$ or $\delta\geq \overline{\delta}$. The possibility distribution $\pi_{\tilde{\delta}}$ can be identified with a membership function~$\mu_{\tilde{\delta}}$
of the following fuzzy interval in $\mathbb{R}$, denoted by $\tilde{\delta}=\braket{0,\overline{\delta}}_{z}$
 (see Figure~\ref{fig1}b):
\begin{equation}
	\mu_{\tilde{\delta}}(x)=\left\{\begin{array}{lllll}
				(1-\frac{x}{\overline{\delta}})^{\frac{1}{z}} & x\in [0,\overline{\delta}]\\
				0 & \text{otherwise}.
			\end{array}\right.
\label{mud}
\end{equation}
A detailed interpretation of the possibility distribution of a real-valued uncertain variable and some methods to obtain it from the possessed knowledge are described in~\cite{DP88}.

Let us now define the following joint possibility distribution for scenarios in $\mathbb{R}^n$ of 
uncertain parameters~$\tilde{\pmb{a}}$ in the constraint~(\ref{uc}):
\begin{equation}
\pi(\pmb{a})=\min\{\pi_{\tilde{a}_1}(a_1),\dots,\pi_{\tilde{a}_n}(a_n), \pi_{\tilde{\delta}}(\delta(\pmb{a}))\}.
\label{jpos}
\end{equation}
The first part of $\pi$ is built from the non-interacting marginals $\pi_{\tilde{a}_1},\dots,\pi_{\tilde{a}_n}$ (see~\cite{DDC09}),
which are possibility distributions for the values of uncertain parameters $\tilde{a}_1,\ldots, \tilde{a}_n$, respectively,
and the second part models possible interactions among the components of $\tilde{\pmb{a}}$. Observe that $\pi$ satisfies all the points of Assumption~\ref{assump1}. In particular, $\pi(\hat{\pmb{a}})=1$ for the nominal scenario~$\hat{\pmb{a}}$. 

Let  $[\underline{a}_j(\lambda),\overline{a}_j(\lambda)]$ and $[\underline{\delta}(\lambda), \overline{\delta}(\lambda)]$
be the $\lambda$-cuts of $\tilde{a}_j$ and $\tilde{\delta}$, i.e.
the intervals that contain the values of $\tilde{a}_j$ and $\tilde{\delta}$
such  $\pi_{\tilde{a}_j}(a_j)\geq \lambda$  and $\pi_{\tilde{\delta}}(\delta(\pmb{a}))\geq \lambda$, $\lambda\in (0,1]$, respectively.
Clearly, the $\lambda$-cut of $\tilde{\pmb{a}}$ is $\mathcal{C}(\lambda)$ (see~(\ref{ccut})).
By~(\ref{jpos}) we get $\pi(\pmb{a})\geq \lambda$ if and only if $\pi_{\tilde{a}_j}(a_j)\geq \lambda$ for each $j\in [n]$ and  $\pi_{\tilde{\delta}}(\delta(\pmb{a}))\geq \lambda$, for any $\lambda\in (0,1]$.
Since $\pi_{\tilde{a}_j}(a_j)\geq \lambda$ if and only if $a_j\in [\underline{a}_j(\lambda),\overline{a}_j(\lambda)]$
and $\pi_{\tilde{\delta}}(\delta(\pmb{a}))\geq \lambda$ if and only if $||\pmb{B}(\pmb{a}-\hat{\pmb{a}})||_{p}\leq \overline{\delta}(\lambda)$ for each $\lambda\in (0,1]$, we obtain an alternative representation of the confidence set $\mathcal{C}(\lambda)$:
\begin{equation}
\mathcal{C}(\lambda)=\{\pmb{a}\in\mathbb{R}^n:\, \underline{\pmb{a}}(\lambda)\leq \pmb{a}\leq \overline{\pmb{a}}(\lambda),||\pmb{B}(\pmb{a}-\hat{\pmb{a}})||_{p}\leq \overline{\delta}(\lambda)\}
\label{ccut1}
\end{equation}
for each $\lambda\in (0,1]$, where  $\underline{\pmb{a}}(\lambda)=(\underline{a}_1(\lambda),\dots,\underline{a}_n(\lambda))$ and $\overline{\pmb{a}}(\lambda)=(\overline{a}_1(\lambda),\dots,\overline{a}_n(\lambda))$
are vectors of bounds of the $\lambda$-cuts of~$\tilde{a}_j$, $j\in[n]$.
For  the membership functions~(\ref{mua}) and (\ref{mud}), the bounds have 
the following forms: 
\begin{align}
\underline{a}_j(\lambda)&=\hat{a}_j-\underline{a}_j(1-\lambda^{z_1}),& 
\overline{a}_j(\lambda)&=\hat{a}_j+\overline{a}_j(1-\lambda^{z_2})&\text{ and }\nonumber\\
\underline{\delta}(\lambda)&=0,&
\overline{\delta}(\lambda)&=\overline{\delta}(1-\lambda^{z}),&  \label{cutszz}
\end{align}
respectively.
We also set $\mathcal{C}(0)=\mathcal{I}\cap\mathcal{D}$.

From now on, the confidence set $\mathcal{C}(\lambda_i)$, $i\in \Lambda$,
in~(\ref{sinfc}) (also in the formulation~(\ref{repruccvar1})-(\ref{repruccvar5}))
will be of the form~(\ref{ccut1}).

\section{Tractable formulations of CVaR constraint}
\label{stfc}

In this section, we will construct tractable  reformulations of~(\ref{repruccvar1})-(\ref{repruccvar5}) for some special cases of the joint possibility distribution $\pi$ for~$\tilde{\pmb{a}}$
proposed in Section~\ref{sec3} (see~(\ref{jpos})).
Recall (see Remark~\ref{rtfc}) that the tractability of~(\ref{repruccvar1})-(\ref{repruccvar5}) depends on coping with the optimization
problem on the right-hand side of the constraint~(\ref{sinfc}) for fixed $\pmb{x}$ and $\lambda_i$, $i\in \Lambda$,
in particular, with its inner problem   with linear objective function:
\begin{equation}
\label{mod030}
\max_{\pmb{a} \in \mathcal{C}(\lambda_i)} \pmb{a}^T\pmb{x}.
\end{equation}
We will use the following observation, which results from the fact that $g$ is a non-decreasing convex function:
\begin{obs}
\label{obsconveq}
Let $\pmb{a}$ be an optimal scenario to~(\ref{mod030}) for
given $\pmb{x}\in \Xset$ and $\lambda_i$, $i\in \Lambda$. Then~$\pmb{a}$ also maximizes the right-hand side of~(\ref{sinfc}) for any non-decreasing convex function $g$.
\end{obs} 

A further analysis of~(\ref{repruccvar1})-(\ref{repruccvar5}) will fall  into two parts. We first assume that $g$ is an identity function, that is, $g(y)=y$. We will focus on the optimization problem~(\ref{mod030}) and explore the properties of 
 the confidence set $\mathcal{C}(\lambda_i)$ (see~(\ref{ccut1})),  which depend on the $L_p$-norm used in the deviation functions~(\ref{disteq}). In the second part, other kinds of~$g$ will be considered.

\subsection{Identity function $g$}
\label{secident}

The idea of constructing a compact, polynomially solvable formulation of~(\ref{repruccvar1})-(\ref{repruccvar5})
consists in using a dual program to~(\ref{mod030}) and applying duality theorems. We will illustrate this approach using three particular cases of the deviation function~(\ref{disteq}), with the~$L_1$, $L_2$, and $L_{\infty}$ norms.

\subsubsection{The $L_{\infty}$-norm}
In order to simplify the presentation, we assume that $\pmb{B}=\pmb{I}\in \Rset^{n\times n}$ in~(\ref{disteq}), where $\pmb{I}$ is the identity matrix (similar reasoning can be used for any matrix $\pmb{B}$). This assumption is reasonable if there are no statistical data for~$\tilde{\pmb{a}}$ and thus there is no knowledge about the correlations among the components of $\tilde{\pmb{a}}$. If some data are available, then the approach based on the $L_2$  norm (described later) can be more appropriate. Observe that one has to provide only the values of $\hat{a}_j, \underline{a}_j, \overline{a}_j$ for each $j\in [n]$ and $\overline{\delta}$,
 together with the shapes of the  possibility distributions $\pi_{\tilde{a}_j}$   and $\pi_{\tilde{\delta}}$ 
 specified by the parameters $z_1$, $z_2$ and~$z$ (see~(\ref{mua}) and (\ref{mud})).
 In the absence of statistical evidence, all these parameters can be estimated by using, for example, experts' opinions.
Under the above assumptions, the deviation function~(\ref{disteq}) takes the following form:
\begin{equation}
\label{disteqinfty}
\delta(\pmb{a})=||\pmb{a}-\hat{\pmb{a}}||_{\infty}=\max_{j\in [n]} |a_j-\hat{a}_j|,
\end{equation}
i.e. it is  the maximum component-wise deviation of~$\pmb{a}$ from the nominal scenario $\hat{\pmb{a}}$. 
 Applying~(\ref{ccut1}) and (\ref{disteqinfty}) we can rewrite~(\ref{mod030}) as
 the following linear program:
 \begin{align}
	\max \; & \pmb{a}^T\pmb{x}\label{ainf1}\\
		 &a_j\leq \overline{a}_j(\lambda_i) &\forall j\in [n] \label{ainf2}\\
		 &-a_j\leq -\underline{a}_j(\lambda_i) &\forall j\in [n]\label{ainf3}\\
		 &-y_j+a_j\leq \hat{a}_j &\forall j\in [n] \label{ainf4}\\
		 &-y_j-a_j\leq -\hat{a}_j &\forall j\in [n] \label{ainf5}\\
		 & y_j \leq \overline{\delta}(\lambda_i)&\forall  j\in [n] \label{ainf6}\\
		 &y_j\geq 0 & \forall j\in [n] \label{ainf7}\\
		&\pmb{a}\in \Rset^n. \label{ainf8}
\end{align}
 The constant bounds 
 $\underline{a}_j(\lambda_i)$, $\overline{a}_j(\lambda_i)$, and $\overline{\delta}(\lambda_i)$ are determined by~(\ref{cutszz}).
The constraints (\ref{ainf2})-(\ref{ainf7}) model the confidence set $\mathcal{C}(\lambda_i)$ (see~(\ref{ccut1})),
in particular (\ref{ainf4})-(\ref{ainf7}) express the deviation~(\ref{disteqinfty}).
The dual program to (\ref{ainf1})-(\ref{ainf8} is as follows:
 \begin{equation}
	\begin{array}{llll}
	\min &\overline{\pmb{a}}(\lambda_i)^T\pmb{\alpha}_i-\underline{\pmb{a}}(\lambda_i)^T	\pmb{\beta}_i
	+ \hat{\pmb{a}}^T (\pmb{\phi}_i-\pmb{\xi}_i) +\overline{\delta}(\lambda_i)\pmb{1}^T\pmb{\gamma}_i\\
			& \begin{array}{lll}
		 \alpha_{ij}-\beta_{ij}+\phi_{ij}-\xi_{ij}= x_j & \forall j\in [n]\\
		\gamma_{ij}-\phi_{ij}-\xi_{ij}\geq 0 & \forall j\in [n] \\
		\pmb{\alpha}_i, \pmb{\beta}_i,\pmb{\phi}_i, \pmb{\xi}_i, \pmb{\gamma}_i \in \Rset_{+}^n,
		\end{array}
	\end{array}
	\label{dainf}
\end{equation}
where $\pmb{\alpha}_i, \pmb{\beta}_i,\pmb{\phi}_i, \pmb{\xi}_i, \pmb{\gamma}_i$
are vectors of dual variables corresponding to the constraints (\ref{ainf2})-(\ref{ainf6}), respectively, and $\pmb{1}=[1,1,\dots,1]^T$.
The weak duality theorem~(see, e.g.,~\cite[Theorem~5.1]{V14}) 
implies that the optimal value of~(\ref{ainf1}) is  bounded above
by the objective function value of~(\ref{dainf}), which gives
the following 
 counterpart  of Theorem~\ref{tuccvar2}
for the $L_{\infty}$-norm.
\begin{cor}
\label{coinfty}
 The constraint~(\ref{uccvar2}) (formulation~(\ref{repruccvar1})-(\ref{repruccvar5})), where $g$ is an identity function,
  for
  confidence set $\mathcal{C}(\lambda_i)$, $i\in \Lambda$, with
  deviation function~(\ref{disteqinfty})
 is equivalent to the following system of linear constraints:
 \begin{align}
&w+\sum_{i\in \Lambda}(\lambda_i-1) v_i\leq (b-t)(1-\epsilon)&&\label{repruccvar1inf}\\
&w-\sum_{j\leq i}  v_j\geq 0 &&\forall  i\in \Lambda \label{repruccvar2nf}\\
&w-\sum_{j\leq i}  v_j+t \geq  \overline{\pmb{a}}(\lambda_i)^T\pmb{\alpha}_i-\underline{\pmb{a}}(\lambda_i)^T\pmb{\beta}_i+ \hat{\pmb{a}}^T (\pmb{\phi}_i-\pmb{\xi}_i) +\overline{\delta}(\lambda_i)\pmb{1}^T\pmb{\gamma}_i&& \forall  i\in \Lambda
	\label{repruccvar3nf}\\
& \alpha_{ij}-\beta_{ij}+\phi_{ij}-\xi_{ij}= x_j &\forall j\in [n],  &\forall  i\in \Lambda \label{repruccvar4nf}\\
&\gamma_{ij}-\phi_{ij}-\xi_{ij}\geq 0 &\forall j\in [n],  &\forall  i\in \Lambda \label{repruccvar5nf}\\
&\pmb{\alpha}_i, \pmb{\beta}_i,\pmb{\phi}_i, \pmb{\xi}_i, \pmb{\gamma}_i \in \Rset_{+}^n&&\forall  i\in \Lambda \label{repruccvar6nf}\\
&v_i\geq 0 &&\forall  i\in \Lambda\label{repruccvar7nf}\\
&w,t\in \Rset.&&	\label{repruccvar8nf}
\end{align}
\end{cor}

\subsubsection{The $L_{1}$-norm}
A method for the $L_{1}$-norm is similar to that for the $L_{\infty}$-norm (see also~\cite{GKZ21a}). Likewise,
we will assume that $\pmb{B}=\pmb{I}\in \Rset^{n\times n}$ in~(\ref{disteq}).  Then
\begin{equation}
\label{disteq1}
\delta(\pmb{a})=||\pmb{a}-\hat{\pmb{a}}||_{1}=\sum_{j\in [n]} |a_j-\hat{a}_j|
\end{equation}
is just the total deviation of $\pmb{a}$ from the nominal scenario $\hat{\pmb{a}}$. The constraint $\delta(\pmb{a})\leq \overline{\delta}$ is called a \emph{continuous budget} imposed on possible scenarios. 
The model for~(\ref{mod030}) is nearly the same as (\ref{ainf1})-(\ref{ainf8}). It is enough to replace
the constraints~(\ref{ainf6}) with
\begin{equation}
\sum_{j\in [n]} y_j \leq \overline{\delta}(\lambda_i).
\label{a16}
\end{equation}
Now (\ref{ainf4}), (\ref{ainf5}), (\ref{ainf7}) and (\ref{a16}) express the deviation~(\ref{disteq1}).
The dual program to (\ref{ainf1})-(\ref{ainf5}), (\ref{a16}), (\ref{ainf7})-(\ref{ainf8})  is as follows:
 $$
	\begin{array}{llll}
	\min &\overline{\pmb{a}}(\lambda_i)^T\pmb{\alpha}_i-\underline{\pmb{a}}(\lambda_i)^T	\pmb{\beta}_i
	+ \hat{\pmb{a}}^T (\pmb{\phi}_i-\pmb{\xi}_i) +\overline{\delta}(\lambda_i)\gamma_i\\
			& \begin{array}{lll}
		 \alpha_{ij}-\beta_{ij}+\phi_{ij}-\xi_{ij}= x_j &\forall j\in [n]\\
		\gamma_{i}-\phi_{ij}-\xi_{ij}\geq 0 &\forall j\in [n] \\
		\pmb{\alpha}_i, \pmb{\beta}_i,\pmb{\phi}_i, \pmb{\xi}_i \in \Rset_{+}^n\\
		\gamma_i \geq 0,
		\end{array}
	\end{array}
$$
where  $\gamma_i$ is a dual variable
 and $\pmb{\alpha}_i, \pmb{\beta}_i,\pmb{\phi}_i, \pmb{\xi}_i$,
are vectors of dual variables.
Again,
by the weak duality theorem~(see, e.g.,~\cite[Theorem~5.1]{V14}) we obtain a counterpart of Theorem~\ref{tuccvar2}
for the $L_{1}$-norm.
\begin{cor}
\label{co1}
 The constraint~(\ref{uccvar2}) (formulation~(\ref{repruccvar1})-(\ref{repruccvar5})), where $g$ is an identity function,
  for
  confidence set $\mathcal{C}(\lambda_i)$, $i\in \Lambda$, with
  deviation function~(\ref{disteq1})
 is equivalent to the following system of linear constraints:
 \begin{align}
&w+\sum_{i\in \Lambda}(\lambda_i-1) v_i\leq (b-t)(1-\epsilon)&&\label{repruccvar11}\\
&w-\sum_{j\leq i}  v_j\geq 0 &&\forall  i\in \Lambda \label{repruccvar21}\\
&w-\sum_{j\leq i}  v_j+t \geq  \overline{\pmb{a}}(\lambda_i)^T\pmb{\alpha}_i-\underline{\pmb{a}}(\lambda_i)^T\pmb{\beta}_i+ \hat{\pmb{a}}^T (\pmb{\phi}_i-\pmb{\xi}_i) +\overline{\delta}(\lambda_i)\gamma_i&& \forall  i\in \Lambda
	\label{repruccvar31}\\
& \alpha_{ij}-\beta_{ij}+\phi_{ij}-\xi_{ij}= x_j &\forall j\in [n], &\forall  i\in \Lambda \label{repruccvar41}\\
&\gamma_{i}-\phi_{ij}-\xi_{ij}\geq 0 &\forall j\in [n], &\forall  i\in \Lambda \label{repruccvar51}\\
&\pmb{\alpha}_i, \pmb{\beta}_i,\pmb{\phi}_i, \pmb{\xi}_i \in \Rset_{+}^n&&\forall  i\in \Lambda \label{repruccvar61}\\
&v_i, \gamma_i \geq 0 &&\forall  i\in \Lambda\label{repruccvar71}\\
&w,t\in \Rset.&&	\label{repruccvar81}
\end{align}
\end{cor}

\subsubsection{The  $L_{2}$-norm}

Suppose that some statistical data (for example, a set of past realizations) are available for $\tilde{\pmb{a}}$. We can then estimate the mean vector $\hat{\pmb{a}}$ and the covariance matrix~$\pmb{\Sigma}\in \Rset^{n\times n}$ for~$\tilde{\pmb{a}}$. Then we can fix $\underline{a}_j=\overline{a}_j=k\cdot \sigma_j=k\cdot\sqrt{\Sigma_{jj}}$ for each $j\in [n]$ and some constant $k$. By the Chebyshev inequality, we can ensure that $a_j\in [\hat{a}_j-\underline{a}_j,\hat{a}+\overline{a}_j]$ with a high probability.
Fix $\pmb{B}=\pmb{\Sigma}^{-\frac{1}{2}}$ (the inverse of square root of~$\pmb{\Sigma}$). Then
\begin{equation}
\label{disteq2}
\delta(\pmb{a})=||\pmb{B}(\pmb{a}-\hat{\pmb{a}})||_{2}.
\end{equation}
The set $\{\pmb{a}\in \Rset^n: \delta(\pmb{a})\leq \overline{\delta}\}$ is an ellipsoid in $\mathbb{R}^n$. Ellipsoidal uncertainty is commonly used in robust optimization (see, e.g.~\cite{BS04r, BN09}).
The model~(\ref{mod030}) is  now 
 \begin{align}
	\max \; & \pmb{a}^T\pmb{x}\label{ainfe1}\\
		 &a_j\leq \overline{a}_j(\lambda_i) &\forall j\in [n] \label{ainfe2}\\
		 &-a_j\leq -\underline{a}_j(\lambda_i) &\forall j\in [n]\label{ainfe3}\\
		 &||\pmb{B}(\pmb{a}-\hat{\pmb{a}})||_{2}\leq \overline{\delta}(\lambda_i) \label{ainfe4}\\
		&\pmb{a}\in \Rset^n. \label{ainfe8}
\end{align}

The model (\ref{ainfe1})-(\ref{ainfe8}) is \emph{second-order cone optimization problem}, whose dual is (see~\cite{GKZ21b}):
\begin{equation}
\label{mod05}
\begin{array}{lll}
	\min & \displaystyle (\overline{\pmb{a}}(\lambda_i)-\hat{\pmb{a}})^T\pmb{\alpha}_i+ 
	(\hat{\pmb{a}}-\underline{\pmb{a}}(\lambda_i))^T\pmb{\beta}_i+
	 \overline{\delta}(\lambda_i)\gamma_i+\hat{\pmb{a}}^T\pmb{x}\\
	& \begin{array}{ll}
		\alpha_{ij} - \beta_{ij} + \pmb{B}^{T}_j\pmb{u}_i=x_j &\forall j\in [n]\\
		 ||\pmb{u}_i||_{2}\leq \gamma_i \\
		 \pmb{\alpha}_i, \pmb{\beta}_i \in \Rset_{+}^n\\
		 \pmb{u}_i\in \Rset^n \\
		 \gamma_i\geq 0
		 \end{array}
\end{array}
\end{equation}
where $\pmb{B}_j$ is the $j$th column of the matrix $\pmb{B}$,
$\gamma_i$ is a dual variable
and $\pmb{\alpha}_i, \pmb{\beta}_i,\pmb{u}_i$
 and dual variable vectors.
 The weak duality (see, e.g.,~\cite{BV08}) implies that the optimal objective function value of~(\ref{mod030}) is bounded above by
 the objective function value of~(\ref{mod05}).
 Hence~(\ref{mod05}) leads to
 a counterpart  of Theorem~\ref{tuccvar2}
for the $L_{2}$-norm.
\begin{cor}
\label{co2}
 The constraint~(\ref{uccvar2}) (formulation~(\ref{repruccvar1})-(\ref{repruccvar5})), where $g$ is an identity function,
 for
  confidence set $\mathcal{C}(\lambda_i)$, $i\in \Lambda$, with
  deviation function~(\ref{disteq2})
 is equivalent to the following system of second-order cone constraints:
 \begin{align}
&w+\sum_{i\in \Lambda}(\lambda_i-1) v_i\leq (b-t)(1-\epsilon)&&\label{repruccvar12}\\
&w-\sum_{j\leq i}  v_j\geq 0 &&\forall  i\in \Lambda \label{repruccvar22}\\
&w-\sum_{j\leq i}  v_j+t \geq  (\overline{\pmb{a}}(\lambda_i)-\hat{\pmb{a}})^T\pmb{\alpha}_i+   
	(\hat{\pmb{a}}-\underline{\pmb{a}}(\lambda_i))^T\pmb{\beta}_i+
	 \overline{\delta}(\lambda_i)\gamma_i+\hat{\pmb{a}}^T\pmb{x} && \forall  i\in \Lambda\label{repruccvar32}    \\	
&\alpha_{ij} - \beta_{ij} + \pmb{B}^{T}_j\pmb{u}_i=x_j &\forall j\in [n],&\forall  i\in\Lambda\label{repruccvar42}\\
&||\pmb{u}_i||_{2}\leq \gamma_i& &\forall  i\in\Lambda\label{repruccvar52} \\
& \pmb{\alpha}_i, \pmb{\beta}_i \in \Rset_{+}^n&&\forall  i\in\Lambda\label{repruccvar62}\\
& \pmb{u}_i\in \Rset^n&&\forall  i\in\Lambda\label{repruccvar72} \\
&v_i, \gamma_i \geq 0 &&\forall  i\in \Lambda\label{repruccvar82}\\
&w,t\in \Rset.&	&\label{repruccvar92}
\end{align}
\end{cor}

\subsection{General convex nondecreasing function $g$}
\label{arbitrg}

Consider the formulation~(\ref{repruccvar1})-(\ref{repruccvar5}) with a given convex nondecreasing function $g$. If  \emph{derivative} or \emph{subderivative} of~$g$ at point~$y_0$,  for each $y_0\in \Rset$
can be provided, then one can apply a polynomial-time ellipsoid algorithm-based approach~\cite{GLS93} to
the constraint~(\ref{uccvar2}) (see~Remark~\ref{rtfc}). To obtain a compact reformulation of the model, one can approximate $g$ by using  a non-decreasing convex piecewise affine function
\begin{equation}
\label{convappr}
g(y)=\max_{z\in \Xi}\{r_z y+ s_z\}, \;z\in \Xi=\{1,\dots,\xi\},
\end{equation} 
where 
$\Xi$ is a finite index set and $s_z\in \Rset$,
$r_z\geq 0$ for each $z\in \Xi$. The quality of the approximation depends on the cardinality of the index set~$\Xi$.

Using Observation~\ref{obsconveq} and the results of Section~\ref{secident}, we can provide compact, polynomially solvable formulations of (\ref{repruccvar1})-(\ref{repruccvar5})
for  deviation functions~(\ref{disteqinfty}),  (\ref{disteq1}) and  (\ref{disteq2}), respectively.
It suffices
to replace constraints~(\ref{repruccvar3nf})  (see Corollary~\ref{coinfty}) with
\begin{align}
w-\sum_{j\leq i}  v_j+t \geq &r_z( \overline{\pmb{a}}(\lambda_i)^T\pmb{\alpha}_i-\underline{\pmb{a}}(\lambda_i)^T\pmb{\beta}_i+&
         \nonumber\\
	&+ \hat{\pmb{a}}^T (\pmb{\phi}_i-\pmb{\xi}_i) +\overline{\delta}(\lambda_i)\pmb{1}^T\pmb{\gamma}_i)+s_z
	\;\; \forall  z\in \Xi,
	& \forall  i\in \Lambda, \label{repruccvar3nfg}
\end{align}
constraints~(\ref{repruccvar31})  (see Corollary~\ref{co1}) with
 \begin{align}
w-\sum_{j\leq i}  v_j+t \geq  &r_z (\overline{\pmb{a}}(\lambda_i)^T\pmb{\alpha}_i-\underline{\pmb{a}}(\lambda_i)^T\pmb{\beta}_i+&
         \nonumber\\
	&+ \hat{\pmb{a}}^T (\pmb{\phi}_i-\pmb{\xi}_i) +\overline{\delta}(\lambda_i)\gamma_i)+s_z
	\;\; \forall  z\in \Xi,
	& \forall  i\in \Lambda
	\label{repruccvar31g}
 \end{align}
 and 
 constraints~(\ref{repruccvar32})  (see Corollary~\ref{co2}) with
 \begin{align}
 w-\sum_{j\leq i}  v_j+t \geq &r_z ((\overline{\pmb{a}}(\lambda_i)-\hat{\pmb{a}})^T\pmb{\alpha}_i+ &  \nonumber\\
	&+(\hat{\pmb{a}}-\underline{\pmb{a}}(\lambda_i))^T\pmb{\beta}_i+
	 \overline{\delta}(\lambda_i)\gamma_i+\hat{\pmb{a}}^T\pmb{x}) +s_z
	 \;\; \forall  z\in \Xi,
	 & \forall  i\in \Lambda,\label{repruccvar32g}
\end{align}
respectively.


\section{Solving optimization problems with uncertain coefficients}
\label{sec4}

In this section, we apply the solution concept described previously to various optimization problems. We will characterize the complexity of the resulting deterministic counterparts.
 Consider the following optimization problem with uncertain constraints and objective function:
\begin{equation}
\label{modunc}
	\begin{array}{lll}
		\widetilde{\min} &  \tilde{\pmb{c}}^T\pmb{x}\\
				 & \tilde{\pmb{a}}_i^T\pmb{x}\leq b_i &\forall i\in [m]\\
				 &\pmb{x}\in \mathbb{X},
	\end{array}
\end{equation}
where $\tilde{\pmb{c}}=(\tilde{c}_{j})_{j\in[n]}$, $\tilde{\pmb{a}}_i=(\tilde{a}_{ij})_{j\in[n]}$, $i\in [m]$, are uncertain objective function and constraint coefficients.  The right-hand constraint sides $b_i \in \Rset$, $i\in [m]$, are assumed to be precise (we will show later how to deal with uncertain right-hand sides).
The uncertainty of each coefficient is modeled by a fuzzy interval whose membership function is regarded as a possibility
distribution for its values (according to the interpretation provided in Section~\ref{sec2}).
The set $\mathbb{X}$ specifies the domain of decision variables $\pmb{x}$.

Using the solution concept proposed in this paper, we can transform~(\ref{modunc}) into the following deterministic equivalent:
\begin{equation}
\label{lpcar}
\begin{array}{lll}
		\min & h \\
		                 &\displaystyle \sup_{{\rm P}\in \mathcal{P}^{\ell}_{\pi_0}}
				 {\rm CVaR}_{\mathrm{P}}^{\epsilon}[\tilde{\pmb{c}}^T\pmb{x}]\leq h&\\
				&\displaystyle \sup_{{\rm P}\in \mathcal{P}^{\ell}_{\pi_i}}
				 {\rm CVaR}_{\mathrm{P}}^{\epsilon}[\tilde{\pmb{a}}_i^T\pmb{x}]\leq b_i &\forall i\in [m]\\
				 &h\in \Rset&\\
				 &\pmb{x}\in \mathbb{X},&				 
\end{array}
\end{equation}
where $\pi_0$ and $\pi_i$, $i\in[m]$, are joint possibility distributions
induced by the possibility distributions of the objective function and the constraint coefficients, respectively, 
according to~(\ref{jpos}). If the right-hand sides of the constraints are also uncertain in~(\ref{modunc}), then
the constraints in~(\ref{lpcar}) can be rewritten as 
$\sup_{{\rm P}\in \mathcal{P}^{\ell}_{\pi_i}}
				 {\rm CVaR}_{\mathrm{P}}^{\epsilon}[\tilde{\pmb{a}_i}^T\pmb{x} -\tilde{b}_i x_0]\leq 0$,
where $x_0$ is an auxiliary variable such that $x_0=1$.			 

The solutions to~(\ref{lpcar}) have the following interpretation. Recall that
the parameter~$\epsilon\in [0,1)$ is provided by the decision-maker
and reflects her/his attitude toward risk.   When~$\epsilon=0$, then CVaR becomes the expectation
(which models risk neutrality). This case leads to solutions~$\pmb{x}$ to~(\ref{lpcar})
 that guarantee a long-run performance.
For a greater value of~$\epsilon$ (more risk aversion is taken into account), more attention is paid to the worst outcomes and
solutions~$\pmb{x}$  to~(\ref{lpcar}) become more robust, i.e., guarantee better worst-case performance.

The model~(\ref{lpcar}) has $m+1$ CVaR constraints, which can be represented as systems of linear or second-order cone constraints (see Section~\ref{stfc}).
Hence, if $\mathbb{X}$ is a polyhedral set described by a system of linear constraints, then~(\ref{lpcar}) becomes
  linear programming problems (for  the $L_{1}$ or $L_{\infty}$ norms) or
 a second order cone optimization problem (for  the $L_{2}$-norm).
 Both cases can be solved efficiently by standard off-the-shelf solvers.
 \begin{cor}
  If $\mathbb{X}$ is a polyhedral set, then~(\ref{lpcar}) is polynomially solvable for the $L_{1}$, $L_{2}$, or $L_{\infty}$ norms, used in~(\ref{disteq}).
 \end{cor}

Consider the following special case of~(\ref{lpcar}) for $m=0$:
\begin{equation}
\label{modunc2}
	\begin{array}{lll}
		\min &  \displaystyle \sup_{{\rm P}\in \mathcal{P}^{\ell}_{\pi_0}}
				 {\rm CVaR}_{\mathrm{P}}^{\epsilon}[\tilde{\pmb{c}}^T\pmb{x}]\\
				 &\pmb{x}\in \mathbb{X}\subseteq \{0,1\}^n.
	\end{array}
\end{equation}
Hence,~(\ref{modunc2}) is a CVaR counterpart  of 
 a combinatorial optimization problem with uncertain costs. 
 If  the $L_{2}$-norm is used to define the deviation function~(\ref{disteq}) and $\epsilon\rightarrow 1$, then~(\ref{modunc2}) becomes the following \emph{robust combinatorial optimization problem}:
\begin{equation}
\label{modunc3}
	\min_{\pmb{x}\in \mathbb{X}}  \max_{\pmb{c}\in \mathcal{U}}  \pmb{c}^T\pmb{x},
\end{equation}
where $\mathcal{U}=\{\pmb{c}: ||\pmb{B}(\pmb{c}-\hat{\pmb{c}})||_{2}\leq \overline{\delta}\}$ is an ellipsoidal uncertainty set. Problem~(\ref{modunc3}) is known to be NP-hard even for some very special cases of $\mathbb{X}$~\cite{BS04r}.
 As a consequence, we get the following:
 \begin{cor}
 If $\mathbb{X}\subseteq \{0,1\}^n$, then~(\ref{modunc2}) is NP-hard for  the~$L_{2}$-norm used in~(\ref{disteq}).
  \end{cor}
 
 Exploring the complexity of~(\ref{modunc2}) with other norms used in the deviation function~(\ref{disteq})
 is an interesting research direction.

\section{Computational examples}
\label{sec5}

In this section, we will show two applications of the models constructed in Section~\ref{stfc}. We first examine a 
continuous knapsack problem in which
the item profits $c_j$, $j\in [n]$, are deterministic and the item weights
$\tilde{a}_j$, $j\in [n]$,  are uncertain. 
Thus, we wish to choose fractions of items in the set $[n]=\{1,\dots,n\}$
 to maximize the total profit, subject to the condition that the total weight (being an uncertain variable)
 is at most
 a given limit~$b$.
 Suppose that we only know the nominal item weights~$\hat{a}_j$ and estimate that the weight 
 realizations $a_j\in [(1-q)\hat{a}_j, (1+q)\hat{a}_j]$ for some prescribed value $q\in [0,1]$. We use  fuzzy intervals 
 $\tilde{a}_j=\braket{\hat{a}_j, (1-q)\hat{a}_j, (1+q)\hat{a}_j}_{0.5-0.5}$ to define the possibility distributions $\pi_{\tilde{a}_j}$ 
 for each $j\in [n]$. The deviation function takes the form of~(\ref{disteq1}) (the $L_1$-norm is used). 
 We set $\overline{\delta}= \Delta\sum_{j\in [n]} \hat{a}_j$ for some fixed value of $\Delta\in [0,1]$ which represents the amount of uncertainty in the item weight realizations. The membership function of the fuzzy interval $\tilde{\delta}=\braket{0,\overline{\delta}}_{1}$ represents the possibility distribution of the uncertain deviation. 
   
   The sample data for the problem are generated as follows. We chose $\hat{a}_i=U([1,100])$, where $U([a,b])$ is a uniformly distributed random number in the interval $[a,b]$. We fix $q=0.4$ and $b=0.3\sum_{j\in [n]} \hat{a}_j$.
   In our example, $g$ is the identity function, and $\ell=100$.
  This gives  the following version of~(\ref{lpcar}): 
 \begin{equation}
\label{unckn}
	\begin{array}{lll}
			\max & \displaystyle \sum_{j\in [n]} c_jx_j \\
				  & \displaystyle \sup_{{\rm P}\in \mathcal{P}^{\ell}_{\pi}}
				 {\rm CVaR}_{\mathrm{P}}^{\epsilon}[\tilde{\pmb{a}}^T\pmb{x}] \leq b\\
				  & 0\leq x_j\leq 1 & \forall j\in [n].
	\end{array}
\end{equation} 

In order to solve~(\ref{unckn}) we used the reformulation~(\ref{repruccvar11})-(\ref{repruccvar81})  (see Corollary~\ref{co1}) of the 
 CVaR constraint, which leads to a linear programming problem.

The results for a sample instance with $n=50$ are shown in Figure~\ref{figex1}. The model has been solved for $\Delta\in \{0, 0.1, 0.2, 0.3\}$ and $\epsilon\in\{0,0.1,\dots,0.9\}$. For $\Delta=0$, there is no uncertainty and the CVaR constraint in~(\ref{unckn}) becomes $\hat{\pmb{a}}^T\pmb{x}\leq b$. Obviously, in this case, the optimal objective value is the same for all $\epsilon$ which is represented in Figure~\ref{figex1} as the horizontal line with $opt=1417$. Increasing $\Delta$, we increase the amount of uncertainty, and, in consequence, the optimal objective value decreases. This can be observed in the subsequent 
curves
 in Figure~\ref{figex1}, below the horizontal line.  One can also observe a deterioration in the objective function values as $\epsilon$ increases. The decision-maker can now choose a suitable combination of $\Delta$ and $\epsilon$ to take into account her/his attitude toward risk.
 \begin{figure}[ht!]
\centering
\includegraphics[width=10cm]{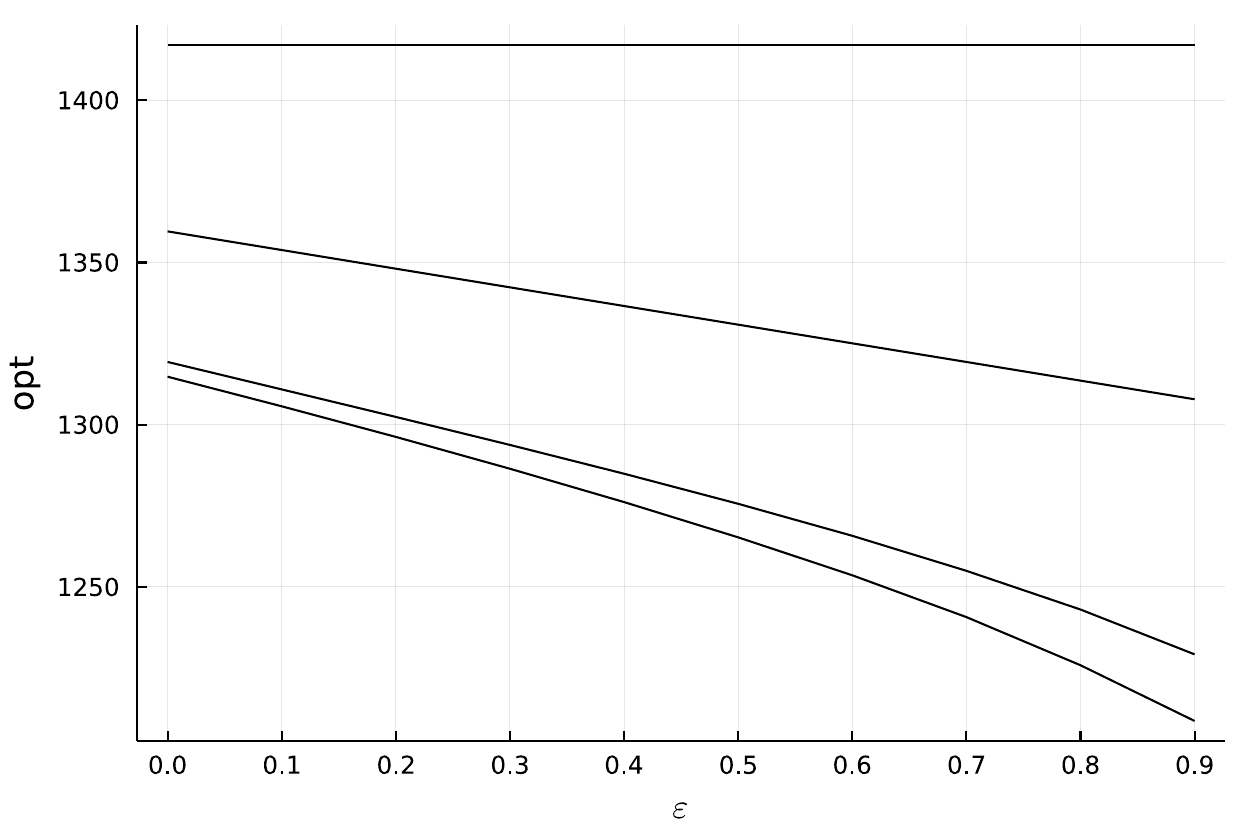}
\caption{The optimal objective function values for various $\epsilon$ and $\Delta=0,0.1,0.2, 0.3$. The horizontal line represents the case with $\Delta=0$, and the next 
curves represent $\Delta=0.1,0.2,0.3$,  respectively. } \label{figex1}
\end{figure}

In the second example, we consider the following portfolio selection problem. We are given $n$ assets. Let $\tilde{\pmb{r}}=(\tilde{r}_1,\dots,\tilde{r}_n)$ be a random vector of the returns of these $n$ assets. Typically, the probability distribution of $\tilde{\pmb{r}}$ is unknown. However, one can estimate the mean $\hat{\pmb{r}}$ and the covariance matrix $\pmb{\Sigma}$ of $\tilde{\pmb{r}}$, for example, using its past realizations. We can now construct a joint possibility distribution for $\tilde{\pmb{r}}$ as follows. For the component $\tilde{r}_j$, we estimate its support $[\hat{r}_j-6\sqrt{\Sigma_{jj}}, \hat{r}_j+6\sqrt{\Sigma_{jj}}]$, where $\Sigma_{jj}=\sigma^2_j$ is the estimated variance of $\tilde{r}_j$. By Chebyshev's inequality, the value of $\tilde{r}_j$ belongs to the support with the high probability. Now, the possibility distribution $\pi_{\tilde{r}_j}$ is a membership function of the fuzzy interval 
$\tilde{r}_j=\braket{\hat{r}_j, \hat{r}_j-6\sqrt{\Sigma_{jj}}, \hat{r}_j+6\sqrt{\Sigma_{jj}}}_{1-1}$ (we use triangular fuzzy intervals; however, one can use other shapes to reflect the uncertainty better). 

The distance function  takes the form:
$\delta(\pmb{r})=||\pmb{\Sigma}^{-\frac{1}{2}}(\pmb{r}-\hat{\pmb{r}})||_{2}$.
Now, the membership function of the fuzzy interval $\tilde{\delta}=\braket{0,\overline{\delta}}_{1}$, for some prescribed~$\overline{\delta}$, represents the possibility distribution $\pi_{\tilde{\delta}}$ for the deviation. 
Let $\mathbb{X}=\{\pmb{x}\in \mathbb{R}^n_{+}: x_1+\dots+x_n=1\}$ be the set of all portfolios. Now, $-\tilde{\pmb{r}}^T\pmb{x}$ is a random variable that denotes the loss of portfolio $\pmb{x}$. The decision-maker can express his loss in terms of some disutility function $g(-\tilde{\pmb{r}}^T\pmb{x})$. A non-linear disutility allows us to express the increasing marginal loss. In the example, we use $g(y)=e^y$. We consider the following version of~(\ref{lpcar}):
\begin{equation}
\label{portfolio}
\begin{array}{ll}
		\min & h \\
		                 &\displaystyle \sup_{{\rm P}\in \mathcal{P}^{\ell}_{\pi}}
				 {\rm CVaR}_{\mathrm{P}}^{\epsilon}[g(-\tilde{\pmb{r}}^T\pmb{x})]\leq h\\
				 &\pmb{x}\in \mathbb{X}, h\in \Rset.\\ 
\end{array}
\end{equation}
In order to solve~(\ref{portfolio}) 
we used the reformulation based on 
the piecewise affine approximation of~$g$, 
presented in Section~\ref{arbitrg}. We approximated $g(y)$ using $\xi=10$ affine pieces. By Corollary~\ref{co2} and~(\ref{repruccvar32g}), model~(\ref{portfolio}) can be converted to a  second-order cone optimization problem.

We solved~(\ref{portfolio}) for a sample $n=6$ assets from the Polish stock market. The assets have been chosen from the banking sector. Using 30 past realizations, we estimated the mean and covariance matrix for $\tilde{\pmb{r}}$:
\begin{eqnarray*}
	\hat{\pmb{r}}&=&[0.01, 0.308, -0.076, 0.239, 0.434, 0.377]\\
\pmb{\Sigma}&=&
\arraycolsep=1.9pt
\left[ \begin{array}{rrrrrrrr}
					6.596 & 2.509 & 4.168 & 2.565 & 2.801 & 1.169 \\
					 &  3.171 & 2.961 & 0.597 & 1.067 & 0.828 \\
					 &	        & 8.765 & 2.676 & 2.977 & 0.976 \\
					 &              &	     & 2.907 & 2.299 & 0.975 \\
					 &		& &			  & 3.150 & 1.05 \\
					 &		&		&        &		& 1.545 
			  \end{array}\right]   
\end{eqnarray*}
We choose $\overline{\delta}\in \{0, 2, 4, 6, 8\}$ and $\epsilon\in \{0, 0.1,
\dots,0.9\}$. The results are shown in Figure~\ref{figex2}. For $\overline{\delta}=0$, there is no uncertainty, and in the optimal portfolio $x_5=1$, the $5$th asset has the smallest loss (the largest return). This is represented by the horizontal line in Figure~\ref{figex2}. We will call $\hat{\pmb{x}}=(0,0,0,0,1,0)$ the nominal portfolio.

 \begin{figure}[ht!]
\centering
\includegraphics[width=10cm]{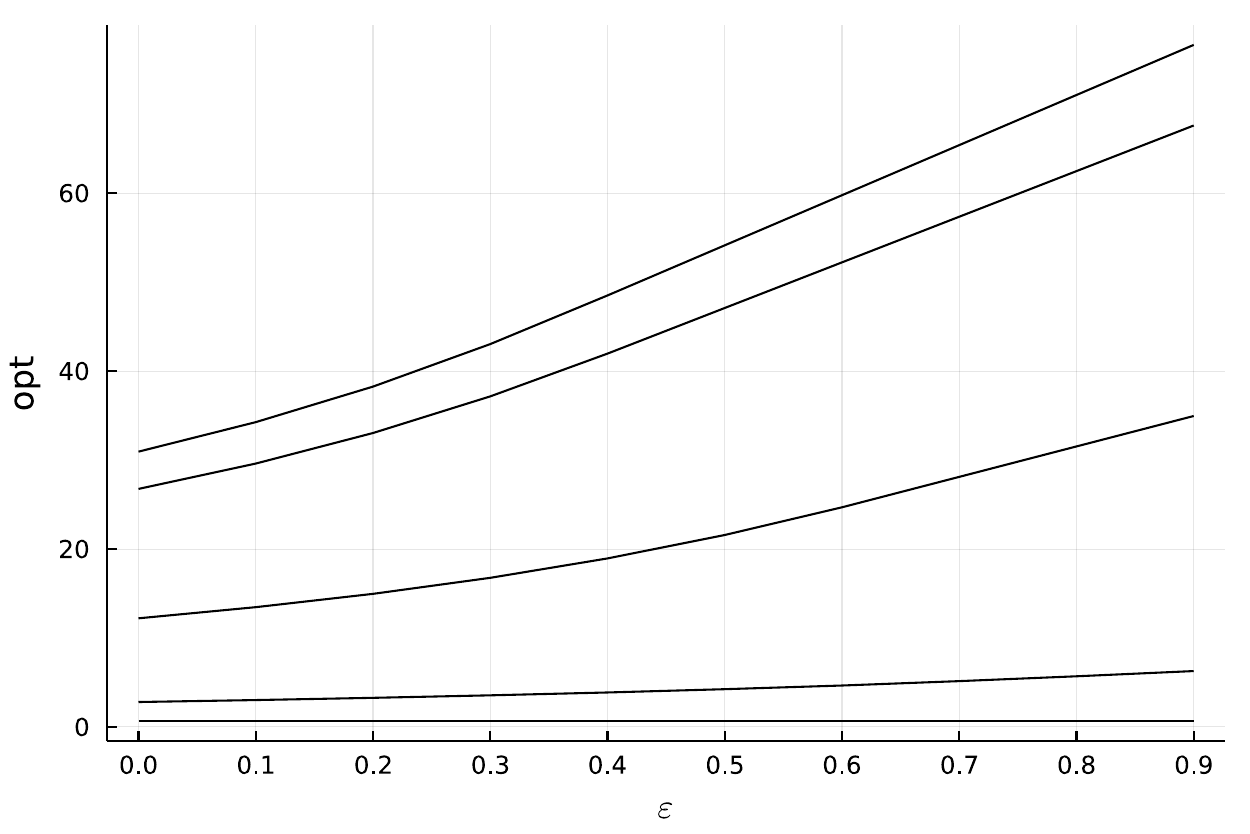}
\caption{The optimal objective function values for various $\epsilon$ and $\overline{\delta}=0, 2, 4, 6, 8$. The horizontal line represents the case with $\overline{\delta}=0$ and the next 
curves
 represent $\overline{\delta}=2, 4, 6, 8$, respectively. } \label{figex2}
\end{figure}

Figure~\ref{figex3} shows a comparison between the value of the objective function of the nominal portfolio $\hat{\pmb{x}}$, denoted by $val$ and the optimal objective value of~(\ref{portfolio}), denoted by $opt$. For a fixed~$\epsilon$, the gap equal to $(val-opt)/opt$ is calculated. One can see that the gap is positive for each $\epsilon$, so taking uncertainty into account allows decision-makers to reduce the risk significantly. The largest gap equal to~1.31 has been reported for $\epsilon=0.4$. In this case, the optimal portfolio is $\pmb{x}=(0,0.2,0,0.16,0.03, 0.61)$. Observe that in this portfolio, diversification is profitable and the largest amount is allocated to asset~$6$, which has a positive return and the smallest variance.

 \begin{figure}[ht!]
\centering
\includegraphics[width=10cm]{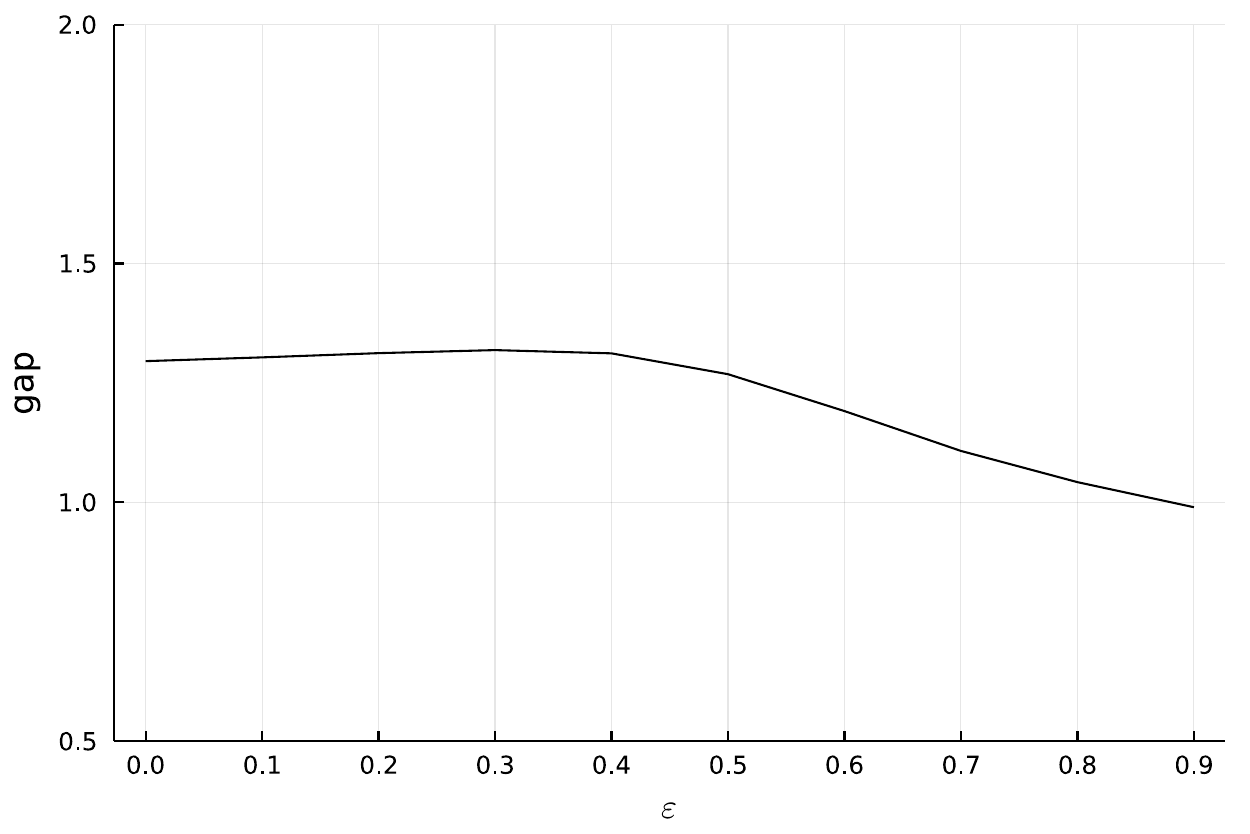}
\caption{The gaps $(val-opt)/opt$ for $\overline{\delta}=4$, where $val$ is the objective function value for the nominal portfolio $\hat{\pmb{x}}$ and $opt$ is the optimal objective value.} \label{figex3}
\end{figure}

\section{Conclusions}

In this paper,  we have proposed a framework for handling uncertain constraints in optimization problems. We have applied fuzzy intervals to model uncertain constraint coefficients. Under the possibilistic interpretation, the membership functions of the fuzzy intervals induce a joint possibility distribution in the set of constraint coefficients realizations (scenario set). Furthermore, this joint possibility distribution characterizes an ambiguity set of probability distributions in the scenario set. We have applied the distributionally robust approach to convert imprecise constraints into deterministic equivalents. For this purpose, we have used the CVaR risk measure, which allows us to take decision-makers' risk aversion into account. We have shown how the concept can be applied to various optimization problems with uncertain constraints and objective functions. Most of the resulting deterministic optimization problems are polynomially solvable, which makes the approach applicable in practice. An interesting research direction is to characterize the problem complexity for the class of combinatorial problems.

\subsubsection*{Acknowledgements}
Romain Guillaume has benefitted from the AI Interdisciplinary
Institute ANITI funding. ANITI is funded by the French Investing for the Future – PIA3 program under
the Grant agreement ANR-19-PI3A-0004.
Adam Kasperski and Pawe{\l} Zieli{\'n}ski were supported by
 the National Science Centre, Poland, grant 2022/45/B/HS4/00355.


\end{document}